\documentclass[11pt]{article}

\usepackage{amsmath, amsthm,amsfonts, amssymb,mathtools,enumerate}
\usepackage{mathrsfs}
\usepackage{stackrel}
\usepackage{tikz}
\usepackage{tkz-euclide}
\usetikzlibrary{snakes}
\usepackage{color}
\usepackage[pdfauthor={ },bookmarksnumbered,hyperfigures,colorlinks=true,citecolor=blue,linkcolor=teal,pdfstartview=FitH]{hyperref}

\usepackage{orcidlink}

\theoremstyle{plain}
\newtheorem{theorem}{Theorem}[section]
\newtheorem{lemma}[theorem]{Lemma}
\newtheorem{proposition}[theorem]{Proposition}

\usepackage[capitalize]{cleveref}
\Crefname{fact}{Fact}{Facts}
\Crefname{claim}{Claim}{Claims}

\theoremstyle{remark}
\newtheorem{remark}[theorem]{Remark}
\newtheorem{definition}[theorem]{Definition}
\newtheorem{question}[theorem]{Question}
\newtheorem{conjecture}[theorem]{Conjecture}

\parskip \medskipamount
\newcommand{\reff}{\mathcal{R}_{\mathrm{eff}}}
\newcommand{\pceff}{\mathcal{C}_{p}}   
\newcommand{\preff}{\mathcal{R}_{p}}   

\newcommand{\be}{\begin{equation}}
	\newcommand{\ee}{\end{equation}}

\begin{document}
\author{
	Pengfei Tang\thanks{Center for Applied Mathematics, Tianjin University.
		Email: \textsf{pengfei\_tang@tju.edu.cn}.}
	~\orcidlink{0000-0003-4268-6015}
}
\date{}
\title{Maximum flow and self-avoiding walk on bunkbed graphs}
\maketitle

\begin{abstract}
 We consider two combinatorial models on bunkbed graphs: maximum flow and self-avoiding walks. A bunkbed graph is defined as the Cartesian product $G\times K_2$, where $G$ is a finite  graph and $K_2$ is the complete graph on two vertices, labelled $0$ and $1$. For the maximum flow problem, we show that if the bunkbed graph $G\times K_2$ has non-negative, reflection-symmetric edge capacities, then for any 
$u, v\in V(G)$, the maximum flow strength from $(u,0)$ to $(v,0)$ in $G\times K_2$ is at least as large as that from $(u,0)$ to $(v,1)$.

For the self-avoiding walk model on a bunkbed graph $G\times K_2$, we investigate whether there are more self-avoiding walks from $(u,0)$ to $(v,1)$ than from $(u,0)$ to $(v,0)$. We prove that this holds when $G=K_n$ is a complete graph and $n$ is sufficiently large. Additionally, we provide examples where the statement does not holds and pose the question of whether it remains true when $\{u,v\}$ is not a cut-edge of $G$.
\end{abstract}

\section{Introduction and main results} \label{sec: intro}

Given a finite (unoriented) graph $G=(V,E)$, the \textbf{bunkbed graph} with base $G$ is defined as the Cartesian product  $G\times K_2$. For simplicity, a vertex $(u,i)\in V(G\times K_2)=V(G)\times \{0,1\}$ will be denoted as $u_i$ for $i\in\{0,1\}$. The edge set $E(G\times K_2)$ consists of two types of  edges: \textbf{horizontal} edges, which correspond to $e\in E(G)\times \{0,1\}$, and \textbf{vertical} edges, which take the form $u_0u_1$ for $u\in V(G)$. 

 The bunkbed conjecture  states that in Bernoulli$(p)$ bond percolation on $G\times K_2$, a vertex $u_0$ is at least as  likely to be connected to $v_0$ as to $v_1$. This conjecture was first proposed---albeit in a slightly different formulation---by Kasteleyn \cite[Remark 5]{Berg_Kahn2001} and was recently proved  false \cite{GPZ2024}. For further details on the final negative resolution of this conjecture, interested readers may refer to \cite{GPZ2024,Hollom2024}.
 
 Beyond percolation, various other random processes on bunkbed graphs have been studied, including random walks \cite{Bollobas_Brightwell1997,Haggstrom1998} and the Ising and random-cluster model \cite{Haggstrom2003probability}.  In the present work, however, we shift our focus from probabilistic models to two combinatorial  models: \textbf{maximum flow} and \textbf{self-avoiding walk} on bunkbed graphs.

\subsection{Maximal flow}

The maximum flow problem is well known, and readers unfamiliar with it may refer to \cite{Ahuja2009} for an introduction.  

Given a finite, unoriented graph \( G = (V,E) \), we consider flows on \( G \) by treating each edge in \( E \) as occurring with both orientations. We denote the set of oriented edges by  
\[
\overrightarrow{E} = \big\{ \overrightarrow{e}, \overleftarrow{e} \colon e \in E \big\}.
\]  
Let \( \mathscr{F}(G) \) be the space of \textbf{antisymmetric} functions \( \theta \) on \( \overrightarrow{E} \), meaning that for each oriented edge \( \overrightarrow{e} \in \overrightarrow{E} \), we have  
\[
\theta(\overrightarrow{e}) = -\theta(\overleftarrow{e}).
\]  
We represent each oriented edge \( \overrightarrow{e} \) as having an arrow directed from its \textbf{tail} \( \overrightarrow{e}^- \) to its \textbf{head} \( \overrightarrow{e}^+ \).  

Define the operator \( d^* : \mathscr{F}(G) \to \mathbb{R}^{V} \) by  
\[
(d^* \theta)(x) := \sum_{\overrightarrow{e} \colon \overrightarrow{e}^- = x} \theta(\overrightarrow{e}).
\]

\begin{definition}\label{def: flow and its strength}
An element $\theta\in\mathscr{F}(G)$ is called a  \textbf{flow} on $G$. 
We say that  $\theta\in\mathscr{F}(G)$ is a flow \textbf{from $A$ to $Z$} if 
\[
(d^*\theta)(x) =
\begin{cases}
0, & \text{if } x \notin A \cup Z, \\
\geq 0, & \text{if } x \in A, \\
\leq 0, & \text{if } x \in Z.
\end{cases}
\]
Here, $A$ is referred to as the \textbf{source}, and $Z$ as the \textbf{sink}. 
If $\theta$ is a flow from $A$ to $Z$, we define its \textbf{strength} as
\[
\mathrm{Strength}(\theta):=\sum_{a\in A} (d^*\theta)(a)\,.
\]
\end{definition}

\begin{definition}\label{def: capacity and admissible flow}
	A \textbf{capacity} is  a nonnegative function $c:E\to [0,\infty]$ defined on the edge set of $G$. We refer to the pair $(G,c)$ a \textbf{capacitated network}. 
	
	Given a finite capacitated network $(G,c)$,   a flow $\theta\in\mathscr{F}(G)$ is said to be \textbf{admissible} for $(G,c)$ if 
	\[
	\big|\theta(\overrightarrow{e})\big| \leq c(e) \quad \text{for all } e \in E\,.
	\]
\end{definition}

\begin{definition}\label{def: maximal flow}
	If $A$ and $Z$ are two disjoint subsets of vertices in $G$, and $c$ is a capacity function on $G$, then the \textbf{maximum flow} from $A$ to $Z$, denoted by $\mathbf{MF}(A,Z)$, is defined as 
	\[
	\mathbf{MF}(A,Z):=\sup\big\{ \mathrm{Strength}(\theta) \,;\, \theta  \mbox{ is an admissible flow from }A \mbox{ to }Z \big\}\,.
	\]
	When $A=\{s\}$ and $Z=\{t\}$ are singletons, we  abbreviate $\mathbf{MF}(\{s\},\{t\})$ as $\mathbf{MF}(s,t)$ .
\end{definition}

Our main results regarding the maximum flow problem on bunkbed graphs is as follows:
\begin{theorem}\label{thm: maximal flow on bunkbed}
	Suppose $c$ is a reflection-symmetric capacity on the bunkbed graph $G\times K_2$, meaning that $c(e_0)=c(e_1)$ for all $e\in E(G)$. Then, for any two vertices $x,y\in V(G)$, we have 
	\[
	\mathbf{MF}(x_0,y_0)\geq \mathbf{MF}(x_0,y_1)\,.
	\]
\end{theorem}
Note that we impose no restrictions on the capacities on the vertical edges of $G\times K_2$ other than non-negativity.

\subsection{Self-avoiding walk}

Given a connected, unoriented graph $G=(V,E)$ and $x\in V$, an $n$-step \textbf{self-avoiding walk} (SAW) $w$ beginning at $x$ is a sequence  $(w_0,e_1,w_1,e_2,w_2,\ldots,w_{n-1},e_n,w_n)$, where $w_0=x$,  $w_j\neq w_k$ for $0\leq j<k\leq n$,  and $e_i$ is an edge connecting $w_{i-1}$ and $w_i$ for each $1\le i\le n$. We denote by $u\sim v$ if there is an edge between $u,v\in V(G)$. 
 We refer to \cite{Bauerschmidt_etal2012lectures_SAW,Madras_Slade1993SAW_book} for background on self-avoiding walks.
While many studies on self-avoiding walks  focused on infinite graphs, such as $\mathbb{Z}^d$,  here we focus on finite bunkbed graphs.

Suppose $G=(V,E)$ is a finite, connected, simple graph. For $x,y\in V$, let $\mathscr{S}(x,y)$ be the set of self-avoiding walks from $x$ to $y$. The following question will be the main focus of our study.
\begin{question}\label{ques: SAW on bunkbed}
	Consider self-avoiding walks on a bunkbed graph $G\times K_2$, where $G$ is a finite, connected, simple graph. Is it true that for all $u\neq v\in V(G)$, we have  
	\[
	\big|\mathscr{S}(u_0,v_0)\big|\leq \big| \mathscr{S}(u_0,v_1) \big|?
	\]
	Here, $\big|\mathscr{S}(\cdot,\cdot)\big|$  denotes the cardinality of $\mathscr{S}(\cdot,\cdot)$. 
\end{question}

Question \ref{ques: SAW on bunkbed} generally has a negative answer, as shown in  Proposition~\ref{prop: SAW on ladder} for the ladder graph. Here we define $P_n\times K_2$ a  \textit{ladder graph}, where $P_n$ is a  \textit{path-graph} with length $n$, i.e., $V(P_n)=\{0,1,\ldots, n\}$ and $E(P_n)=\big\{ (i,i+1)\colon i=0,\ldots,n-1 \big\}$.

\begin{proposition}\label{prop: SAW on ladder}
	Consider self-avoiding walks on the bunkbed graph \( P_n \times K_2 \).  
	
	If \( n \in \{2, 3\} \), then we have 
	\[
	\big|\mathscr{S}(u_0, v_0)\big| = \big|\mathscr{S}(u_0, v_1)\big| \quad \text{for all} \quad u, v \in V(P_n).
	\]  
	
	If \( n \geq 4 \), then 
	\[
	\big|\mathscr{S}(u_0, v_0)\big| > \big|\mathscr{S}(u_0, v_1)\big| \quad \text{for all} \quad u, v \in V(P_n) \text{ with } u \sim v \text{ and neither } u \text{ nor } v \text{ is an endpoint of } P_n,
	\]  
	and
	\[
	\big|\mathscr{S}(u_0, v_0)\big| = \big|\mathscr{S}(u_0, v_1)\big| \quad \text{otherwise.}
	\]
\end{proposition}

If we restrict ourselves to the case of \( K_n \times K_2 \), then we obtain the following partial positive answer to Question \ref{ques: SAW on bunkbed}.
\begin{theorem}\label{thm: SAW on complete}
Consider self-avoiding walks on a bunkbed graph \( K_n \times K_2 \) with \( n \geq 2 \).  
Then, for all distinct vertices \( u, v \in V(K_n) \), we have the following:

\[
\begin{array}{ll}
\big|\mathscr{S}(u_0, v_0)\big| = \big|\mathscr{S}(u_0, v_1)\big| & \text{if } n = 2, \\
\big|\mathscr{S}(u_0, v_0)\big| < \big|\mathscr{S}(u_0, v_1)\big| & \text{if } n \in \{ 3, 4, 5 \}, \\
\big|\mathscr{S}(u_0, v_0)\big| < \big|\mathscr{S}(u_0, v_1)\big| & \text{if } n \text{ is sufficiently large.}
\end{array}
\]

\end{theorem}

\begin{conjecture}
	The inequality $\big|\mathscr{S}(u_0,v_0)\big|\leq \big| \mathscr{S}(u_0,v_1) \big|$ in \cref{thm: SAW on complete} holds for $K_n\times K_2$ for all $n\ge3$. 
\end{conjecture}

\subsection{Organization of the paper}

We prove \cref{thm: maximal flow on bunkbed} in Section~\ref{sec: maximal flow}. To do so, we first generalize the known result for effective resistance on bunkbed graphs \cite{Bollobas_Brightwell1997} to \( p \)-resistance. Next, we apply the max-flow min-cut theorem and a linear programming formulation of the min-cut problem to show that \cref{thm: maximal flow on bunkbed} is essentially a special case of the \( p \)-resistance problem.

For self-avoiding walks on bunkbed graphs, we prove Proposition~\ref{prop: SAW on ladder} and \cref{thm: SAW on complete} in Section~\ref{sec: SAW}. A key idea behind these results is to construct bijections between certain collections of self-avoiding walks on the bunkbed graphs.



\section{Maximal flow on bunkbed graphs}\label{sec: maximal flow}

As stated in the introduction, our strategy for tackling the maximum flow problem is to first consider the $p$-resistance problem.

\subsection{$p$-resistance}
Given a finite, connected, and unoriented graph \( G = (V, E) \), we endow a nonnegative function \( r : E \to (0, \infty) \), and call \( r(e) \) the resistance of the edge \( e \). For such an electrical network \( (G, r) \), the effective resistance between two distinct vertices \( x, y \in V \) is defined as
\[
\reff(x, y) := \inf \left\{ \sum_{e \in E} r(e) \theta(e)^2 \colon \theta \text{ is a unit flow from } x \text{ to } y \right\}.
\]
We use the convention that $\reff(x,x)=0$.

Bollob\'{a}s and Brightwell \cite{Bollobas_Brightwell1997} proved that for bunkbed graphs \( G \times K_2 \), the effective resistance satisfies
\[
\reff(x_0, y_1) \geq \reff(x_0, y_0), \quad \text{for any } x, y \in V(G).
\]
A natural extension of the effective resistance is the following \( p \)-resistance for any fixed \( p > 1 \) (see, for instance, \cite{Soardi1994}, p. 176-178):
\[
\preff(x, y) := \inf \left\{ \sum_{e \in E} r(e) \theta(e)^p \colon \theta \text{ is a unit flow from } x \text{ to } y \right\}.
\]

Define
\[
\pceff(x, y) := \inf \left\{ \sum_{e = uv \in E} \frac{|f(u) - f(v)|^{1 + \frac{1}{p - 1}}}{r_e^{\frac{1}{p - 1}}} \colon f(x) - f(y) = 1 \right\},
\]
where the infimum is taken over all functions \( f : V \to \mathbb{R} \) such that \( f(x) - f(y) = 1 \). There is a useful dual formulation of the \( p \)-resistance, and the following version is given in Proposition 4 of \cite{Alamgir_Luxburg2011phase}.

\begin{lemma}\label{lem: dual form of p-resistance}
 For a fixed $p>1$ and a finite electrical network $(G,r)$, we have 
 \[
 \preff(x,y)=\bigg(\pceff(x,y)\bigg)^{-\frac{1}{p-1}}\,.
 \]
\end{lemma}

Our first result is an extension of the aforementioned inequality \( \reff(u_0, v_1) \geq \reff(u_0, v_0) \) on bunkbed graphs to the general case of \( p \)-resistance.
\begin{theorem}\label{thm: p-resistance}
Suppose \( (G \times K_2, r) \) is a finite electrical network on bunkbed graphs with reflection-symmetric resistance (i.e., \( r(e_0) = r(e_1) \) for all \( e \in E(G) \)) and \( p > 1 \) is a fixed parameter. Then, for any two distinct vertices \( x, y \in V(G) \),
\[
\preff(x_0, y_1) \geq \preff(x_0, y_0).
\]
	
\end{theorem}
We will need a lemma concerning convex functions.
\begin{lemma}\label{lem: an inequality for convex functions}
	Let $\varphi:\mathbb{R}\to\mathbb{R}$ be a convex function. Let $a_0,b_0,a_1,b_1$ be four arbitrary real numbers. Then, we have the following inequality: 
	\be\label{eq: an inequality for convex functions}
	\varphi(a_0-b_0)+\varphi(a_1-b_1)\ge \varphi(a_0\vee a_1-b_0\vee b_1)+\varphi(a_0\wedge a_1-b_0\wedge b_1)\,,
	\ee 
	where $a\vee b:=\max\{a,b\}$ and $a\wedge b:=\min\{a,b\}$.
	
\end{lemma}
\begin{proof}
If \( a_0 = a_1 \) or \( b_0 = b_1 \), then equality obviously holds in \eqref{eq: an inequality for convex functions}. In the following, we assume \( a_0 \neq a_1 \) and \( b_0 \neq b_1 \), and we divide the analysis into four cases:
\begin{itemize}
	\item[(1)] If \( a_0 > a_1 \) and \( b_0 > b_1 \), then equality obviously holds in \eqref{eq: an inequality for convex functions}.
	\item[(2)] If \( a_0 < a_1 \) and \( b_0 < b_1 \), then equality also obviously holds in \eqref{eq: an inequality for convex functions}.
	\item[(3)] If \( a_0 > a_1 \) and \( b_0 < b_1 \), then with \( \lambda = \frac{a_0 - a_1}{(a_0 - a_1) + (b_1 - b_0)} \in (0, 1) \), we have
	\[
	a_0 - b_1 = \lambda(a_0 - b_0) + (1 - \lambda)(a_1 - b_1) \quad \text{and} \quad a_1 - b_0 = (1 - \lambda)(a_0 - b_0) + \lambda(a_1 - b_1).
	\]
	Since \( \varphi \) is convex, by Jensen's inequality, we obtain
	\[
	\varphi(a_0 - b_1) \leq \lambda \varphi(a_0 - b_0) + (1 - \lambda) \varphi(a_1 - b_1) \quad \text{and} \quad
	\varphi(a_1 - b_0) \leq (1 - \lambda) \varphi(a_0 - b_0) + \lambda \varphi(a_1 - b_1).
	\]
	Adding these two inequalities gives the desired inequality \eqref{eq: an inequality for convex functions} in this case.
	\item[(4)] If \( a_0 < a_1 \) and \( b_0 > b_1 \), then this case follows from case (3) by interchanging the roles of \( a_i, b_i \) for \( i \in \{0, 1\} \). \qedhere
\end{itemize}

\end{proof}

\begin{proof}[Proof of \cref{thm: p-resistance}]

In light of Lemma~\ref{lem: dual form of p-resistance}, it suffices to show that
\[
\pceff(x_0,y_1) \leq \pceff(x_0,y_0).
\]
In the definition of \( \pceff(x,y) \), we can assume the infimum is taken over functions \( f \) such that \( f(x) = 1 \) and \( f(y) = 0 \). Furthermore, we can also assume that the functions \( f \) take values in \( [0, 1] \). This is due to the observation that one can modify a function \( f \) to the function \( (f \wedge 1) \vee 0 \) without increasing the corresponding summation
\[
\sum_{e=uv \in E} \frac{|f(u) - f(v)|^{1 + \frac{1}{p-1}}}{r_e^{\frac{1}{p-1}}}.
\]

Now, to prove \( \pceff(x_0, y_1) \leq \pceff(x_0, y_0) \), it suffices to show that for any function \( f: V(G \times K_2) \to [0, 1] \) such that \( f(x_0) = 1 \) and \( f(y_0) = 0 \), one can find a function \( h: V(G \times K_2) \to [0, 1] \) such that
\begin{itemize}
	\item[(a)] \( h(x_0) = 1 \) and \( h(y_1) = 0 \), 
	\item[(b)] \( |h(u_0) - h(u_1)| = |f(u_0) - f(u_1)| \) for any \( u \in V(G) \), and 
	\item[(c)] for any edge \( e = uv \in E(G) \), 
	\[
	|f(u_0) - f(v_0)|^{1 + \frac{1}{p-1}} + |f(u_1) - f(v_1)|^{1 + \frac{1}{p-1}}
	\geq |h(u_0) - h(v_0)|^{1 + \frac{1}{p-1}} + |h(u_1) - h(v_1)|^{1 + \frac{1}{p-1}}.
	\]
\end{itemize}

Our choice of \( h: V(G \times K_2) \to [0, 1] \) is given as follows: for any \( u \in V(G) \),
\[
h(u_0) := f(u_0) \vee f(u_1), \quad \text{and} \quad h(u_1) := f(u_0) \wedge f(u_1).
\]
Now, the requirements (a) and (b) obviously hold for our choice of \( h \). Next, we observe that the function \( \varphi: \mathbb{R} \to \mathbb{R} \) defined by \( \varphi(x) = |x|^q \), where \( q = 1 + \frac{1}{p-1} > 1 \), is convex. The requirement (c) then follows immediately from Lemma~\ref{lem: an inequality for convex functions} with \( a_0 = f(u_0), a_1 = f(u_1), b_0 = f(v_0), b_1 = f(v_1) \).
\end{proof}

\subsection{Proof of \cref{thm: maximal flow on bunkbed}}

The key idea here is the following well-known linear programming formulation of the maximum flow problem. 
\begin{lemma}\label{lem: lpf maxflow}
Suppose $G = (V, E)$ is a finite and connected graph, and $c: E \to (0, \infty)$ is a capacity function. Let $x, y \in V(G)$ be two distinct vertices of $G$. Then the maximum flow between $x$ and $y$ satisfies
\be\label{eq: lpf maxflow}
\mathbf{MF}(x, y) = \min \left\{ \sum_{e = uv \in E} c(e) | f(u) - f(v) | : f(x) = 1, f(y) = 0 \right\}\,.
\ee

\end{lemma}
\begin{proof}
We sketch a proof for the readers' convenience. First, we assume that $c(e) \in \mathbb{N}$ for every edge $e \in E$. Let each edge $e \in E$ be replaced by $c(e)$ parallel edges with capacity $1$, resulting in a new graph $G'$. In the new graph $G'$, every edge has capacity $1$. By the max-flow min-cut theorem, it suffices to show that in $G'$, the following identity holds:
\[
\min \# \left\{ \text{edge-cut set separating } x \text{ from } y \right\} = \min \left\{ \sum_{e = uv \in E(G')} |f(u) - f(v)| : f(x) = 1, f(y) = 0 \right\}\,.
\]
On the one hand, recall Menger's theorem: the minimum edge-cut set separating $x$ from $y$ equals the maximum number of edge-disjoint paths from $x$ to $y$. Along each edge-disjoint path $\gamma$ from $x$ to $y$, the sum of the terms $\sum_{e = uv \in \gamma} |f(u) - f(v)|$ is at least $1$ by the triangle inequality. Hence, we have
\[
\min \# \left\{ \text{edge-cut set separating } x \text{ from } y \right\} \leq \min \left\{ \sum_{e = uv \in E(G')} |f(u) - f(v)| : f(x) = 1, f(y) = 0 \right\}\,.
\]
On the other hand, let $\Pi$ be a minimal (with respect to inclusion) edge-cut set separating $x$ from $y$, and let $S$ be the connected component containing $x$. Set $f(u) = \mathbf{1}_{u \in S}$. Using this particular $f$, we obtain the other direction:
\[
\min \# \left\{ \text{edge-cut set separating } x \text{ from } y \right\} \geq \min \left\{ \sum_{e = uv \in E(G')} |f(u) - f(v)| : f(x) = 1, f(y) = 0 \right\}\,.
\]
Combining the two directions, we conclude that \eqref{eq: lpf maxflow} holds in the case of integer-valued capacities.

Next, observe that if we multiply every edge capacity by the same constant $\alpha$, then the two sides of \eqref{eq: lpf maxflow} will also be multiplied by $\alpha$. Thus, the case of rational-valued capacities follows from the integer-valued case and this observation.

Finally, the general case with real-valued capacities can be proved by a continuity argument, which we omit here.
\end{proof}

\begin{proof}[Proof of \cref{thm: maximal flow on bunkbed}]
The desired inequality $\mathbf{MF}(x_0, y_0) \geq \mathbf{MF}(x_0, y_1)$  follows from Lemma~\ref{lem: lpf maxflow} and the following inequality:
\begin{align*}
&\min\Big\{ \sum_{e = uv \in E(G \times K_2)} c(e) |f(u) - f(v)| : f(x_0) = 1, f(y_0) = 0 \Big\} \\
\ge & \min\Big\{ \sum_{e = uv \in E(G \times K_2)} c(e) |f(u) - f(v)| : f(x_0) = 1, f(y_1) = 0 \Big\}\,,
\end{align*}
which can be proved in the same way as in the proof of \cref{thm: p-resistance}, by using the convex function $\varphi(x) = |x|$ instead.
\end{proof}

\section{Self-avoiding walks on bunkbed graphs}\label{sec: SAW}

Consider self-avoiding walks on a bunkbed graph $G \times K_2$, where $G$ is a finite, connected, simple graph. We first construct bijections between certain collections of self-avoiding walks. 

Fix two distinct vertices $u, v \in V(G)$ and denote the two vertical edges by $e_u = u_0u_1$ and $e_v = v_0v_1$. We then examine the following subsets of $\mathscr{S}(u_0, v_0)$ and $\mathscr{S}(u_0, v_1)$.

\begin{enumerate}
	\item Let $\mathscr{S}_1(u_0,v_0)$ be the subset of walks in $\mathscr{S}(u_0,v_0)$ that do not pass through $v_1$, written as
	\[
	\mathscr{S}_1(u_0,v_0):=\{w \in \mathscr{S}(u_0,v_0) \colon u_0\rightsquigarrow v_0 \textnormal{ not via }v_1 \}\,.
	\]
	Let  $\mathscr{S}_1(u_0,v_1)$ be the subset of walks in $\mathscr{S}(u_0,v_1)$ that
	use the edge $e_v=v_0v_1$ as their last step, written as 
	\[
	\mathscr{S}_1(u_0,v_1):=\{ w \in \mathscr{S}(u_0,v_1) \colon u_0\rightsquigarrow v_0\oplus v_1 \}.
	\]
	
	\item Let $\mathscr{S}_2(u_0,v_0)$ be the subset of walks in $\mathscr{S}(u_0,v_0)$ that use $e_v$ as their last step, written as 
	\[
		\mathscr{S}_2(u_0,v_0):=\{w \in \mathscr{S}(u_0,v_0) \colon u_0\rightsquigarrow v_1\oplus v_0  \}.
	\]
	 Let $\mathscr{S}_2(u_0,v_1)$ be the subset of walks in $\mathscr{S}(u_0,v_1)$ that do not pass through $v_0$, written as
	\[
	\mathscr{S}_2(u_0,v_1):=\{w \in \mathscr{S}(u_0,v_0) \colon u_0\rightsquigarrow v_1 \textnormal{ not via }v_0 \}.
	\]  
	
	\item Let $\mathscr{S}_3(u_0,v_0)$ be the subset of walks in $\mathscr{S}(u_0,v_0)$ that pass though $v_1$ but do not pass through $u_1$ or $e_v$, written as 
	\[
	\mathscr{S}_3(u_0,v_0):=\{w \in \mathscr{S}(u_0,v_0) \colon u_0\rightsquigarrow v_1\rightsquigarrow v_0 \textnormal{ not via }u_1,e_v \}.
	\]
	Let  $\mathscr{S}_3(u_0,v_1)$ be the subset of walks in $\mathscr{S}(u_0,v_1)$ that
	use the edge $e_u=u_0u_1$ as their first step and pass though $v_0$ but do not pass through $e_v$, written as 
	\[
	\mathscr{S}_3(u_0,v_1):=\{ w \in \mathscr{S}(u_0,v_1) \colon u_0\oplus u_1\rightsquigarrow v_0\rightsquigarrow v_1  \textnormal{ not via }e_v\}.
	\]
	
	\item Let $\mathscr{S}_4(u_0,v_0)$ be the subset of walks in $\mathscr{S}(u_0,v_0)$ that use edge $e_u=u_0u_1$  as their first step and pass through $v_1$ but not through $e_v$, written as 
	\[
	\mathscr{S}_4(u_0,v_0):=\{w \in \mathscr{S}(u_0,v_0) \colon u_0\oplus u_1\rightsquigarrow v_1\rightsquigarrow v_0 \textnormal{ not via }e_v\}.
	\]
	
	Let $\mathscr{S}_4(u_0,v_1)$ be the subset of walks in $\mathscr{S}(u_0,v_1)$ that pass through $v_0$ but not  through $u_1$ or $e_v$, written as 
	\[
	\mathscr{S}_4(u_0,v_1):=\{w \in \mathscr{S}(u_0,v_1) \colon u_0\rightsquigarrow v_0\rightsquigarrow v_1  \textnormal{ not via }u_1,e_v\}.
	\]

	\item Let  $\mathscr{S}_5(u_0,v_0)$ be the subset of walks in $\mathscr{S}(u_0,v_0)$ that pass through $u_1$ and $v_1$ but not through the vertical edges $e_u$ or $e_v$, written as 
	\[
	\mathscr{S}_5(u_0,v_0):=\{w \in \mathscr{S}(u_0,v_0) \colon u_0\rightsquigarrow u_1,v_1\rightsquigarrow v_0  \textnormal{ not via } e_u,e_v \}.
	\]
	
	Let $\mathscr{S}_5(u_0,v_1)$ be the subset of walks in $\mathscr{S}(u_0,v_1)$ that pass through $u_1$ and $v_0$ but not through the vertical edges $e_u$ or $e_v$, written as 
	\[
	\mathscr{S}_5(u_0,v_1):=\{w \in \mathscr{S}(u_0,v_1) \colon u_0\rightsquigarrow u_1,v_0\rightsquigarrow v_1  \textnormal{ not via } e_u,e_v \}.
	\]
	Note that in $\mathscr{S}_5(u_0,v_0)$, both possibilities of $u_0\rightsquigarrow u_1\rightsquigarrow v_1\rightsquigarrow v_0 $ and $u_0\rightsquigarrow v_1\rightsquigarrow u_1\rightsquigarrow v_0 $ are allowed. Similarly, both possibilities are allowed for $\mathscr{S}_5(u_0,v_1)$.
\end{enumerate}

It is straightforward to verify that for \( j = 1, 2 \), the sets \( \mathscr{S}_i(u_0, v_j),\,i\in\{1,\ldots,5\} \) are disjoint, and \( \mathscr{S}(u_0, v_j) \) is their union, i.e.,  
\[
\mathscr{S}(u_0, v_j) = \bigcup_{i=1}^{5} \mathscr{S}_i(u_0, v_j).
\]

\begin{lemma}\label{lem: bijections between some subsets of SAW}
For \( i = 1, 2, 3, 4 \), there exists a bijection between \( \mathscr{S}_i(u_0, v_0) \) and \( \mathscr{S}_i(u_0, v_1) \).	
\end{lemma}
\begin{proof}
	We construct the corresponding maps and leave it to the reader to verify that they are indeed bijections.
	
	For \( i = 1,2 \), the map consists of simply adding or deleting the edge \( e_v \).
	
	For \( i = 3 \), given a path \( w \in \mathscr{S}_3(u_0, v_0) \), perform the following operation to obtain a path in \( \mathscr{S}_3(u_0, v_1) \): first, reflect the path \( w \) to obtain a path from \( u_1 \) to \( v_0 \) to \( v_1 \) that does not pass through \( u_0 \) or the edge \( e_v \), and then attach at the beginning the edge \( e_u \).
	
	For \( i = 4 \), given a path \( w \in \mathscr{S}_4(u_0, v_0) \), perform the following operation to obtain a path in \( \mathscr{S}_4(u_0, v_1) \): first, remove the initial edge \( e_u \), and then reflect the resulting path.
	\end{proof}

\subsection{Proof of Proposition~\ref{prop: SAW on ladder}}

\begin{proof}[Proof of Proposition \ref{prop: SAW on ladder}]
The cases \( n = 2,3 \) are straightforward to verify and are therefore omitted. We now assume \( n \geq 4 \).

By Lemma \ref{lem: bijections between some subsets of SAW}, it suffices to compare \( \mathscr{S}_5(u_0, v_0) \) with \( \mathscr{S}_5(u_0, v_1) \).

\begin{enumerate}
	\item[(a)] 
If either \( u \) or \( v \) is an endpoint of \( P_n \), then both \( \mathscr{S}_5(u_0, v_0) \) and \( \mathscr{S}_5(u_0, v_1) \) are clearly empty. Hence, by Lemma \ref{lem: bijections between some subsets of SAW}, it follows that  
\[
\big| \mathscr{S}(u_0, v_0) \big| = \big| \mathscr{S}(u_0, v_1) \big|
\]
in such cases.

	\item[(b)] 	
If neither \( u \) nor \( v \) is an endpoint of \( P_n \) and \( u \sim v \), say \( u = k \) and \( v = k+1 \) for some \( k \in [1, n-2] \), then it is straightforward to verify that \( \mathscr{S}_5(u_0, v_1) \) is empty, while  
\[
\big| \mathscr{S}_5(u_0, v_0) \big| = k(n - k - 1) > 0.
\]
Thus, by Lemma \ref{lem: bijections between some subsets of SAW}, it follows that  
\[
\big| \mathscr{S}(u_0, v_0) \big| > \big| \mathscr{S}(u_0, v_1) \big|
\]
in such cases.

	\item[(c)] 
If neither \( u \) nor \( v \) is an endpoint of \( P_n \) and \( u \not\sim v \), then without loss of generality, we assume \( u = k \) and \( v = k + m \) for some \( 1 \leq k < k + m \leq n - 1 \) and \( m \geq 2 \).  

If \( m = 2 \), then it is straightforward to verify (see Fig. \ref{fig: SAW on ladder 1}) that  
\[
\big| \mathscr{S}_5(u_0, v_0) \big| = \big| \mathscr{S}_5(u_0, v_1) \big| = k(n - k - 2).
\]  

If \( m > 2 \), then it is similarly straightforward to verify (see Fig. \ref{fig: SAW on ladder 2}) that  
\[
\big| \mathscr{S}_5(u_0, v_0) \big| = k(n - k - m) f_0(m), \quad  
\big| \mathscr{S}_5(u_0, v_1) \big| = k(n - k - m) f_1(m),
\]
where \( f_i(m) \) represents the number of self-avoiding walks from \( (0,0) \) to \( (m-2,i) \) in \( P_{m-2} \times K_2 \). By Case (a) above, we have \( f_0(m) = f_1(m) \). Hence,  
\[
\big| \mathscr{S}_5(u_0, v_0) \big| = \big| \mathscr{S}_5(u_0, v_1) \big|.
\]  
Thus, by Lemma \ref{lem: bijections between some subsets of SAW}, it follows that  
\[
\big| \mathscr{S}(u_0, v_0) \big| = \big| \mathscr{S}(u_0, v_1) \big|
\]
in such cases.

	\begin{figure*}[h!]
		\centering
		
		\begin{tikzpicture}[scale=0.8, text height=1.5ex,text depth=.25ex]

		\draw [help lines,dashed] (0,0) grid (11,1);
		\draw [help lines,dashed] (0,4) grid (11,5);
			
		\draw[fill=black] (5,4) circle [radius=0.05]; 
		\node[below] at (5,4-0.1) {$u_0$};
		\draw[fill=black] (5,5) circle [radius=0.05]; 
		\node[above] at (5,5+0.1) {$u_1$};
		
		\draw[fill=black] (7,4) circle [radius=0.05]; 
		\node[below] at (7,4-0.1) {$v_0$};

		\draw[fill=black] (7,5) circle [radius=0.05]; 
		\node[above] at (7,5+0.1) {$v_1$};
		
		\draw[color=red,thick] (5,4)--(2,4)--(2,5)--(9,5)--(9,4)--(7,4);

		\draw[fill=black] (5,0) circle [radius=0.05]; 
		\node[below] at (5,-0.1) {$u_0$};
		\draw[fill=black] (5,1) circle [radius=0.05]; 
		\node[above] at (5,1+0.1) {$u_1$};
		
		\draw[fill=black] (7,0) circle [radius=0.05]; 
		\node[below] at (7,-0.1) {$v_0$};

		\draw[fill=black] (7,1) circle [radius=0.05]; 
		\node[above] at (7,1+0.1) {$v_1$};

		\draw[color=blue,thick] (5,0)--(2,0)--(2,1)--(5,1)--(6,1)--(6,0)--(9,0)--(9,1)--(7,1);

		\end{tikzpicture}
		\caption{Typical paths in $\mathscr{S}_5(u_0,v_0)$ (red) and $\mathscr{S}_5(u_0,v_1)$ (blue) in the case $v=u+2$.}
		\label{fig: SAW on ladder 1}
	\end{figure*}
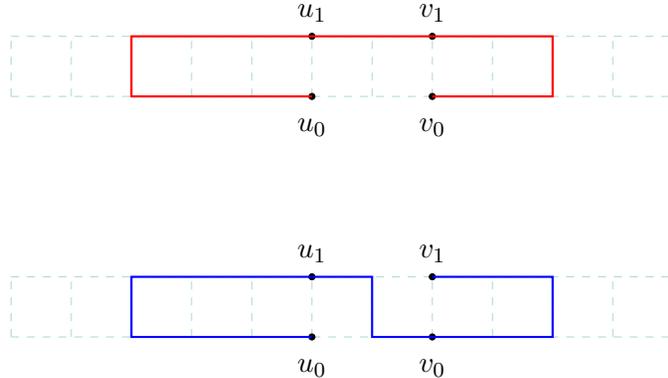	 
	 
	 	\begin{figure*}[h!]
	 	\centering
	 	
	 	\begin{tikzpicture}[scale=0.8, text height=1.5ex,text depth=.25ex]

	 	\draw [help lines,dashed] (0,0) grid (15,1);
	 	\draw [help lines,dashed] (0,4) grid (15,5);
	 	
	 	\draw[fill=black] (5,4) circle [radius=0.05]; 
	 	\node[below] at (5,4-0.1) {$u_0$};
	 	\draw[fill=black] (5,5) circle [radius=0.05]; 
	 	\node[above] at (5,5+0.1) {$u_1$};
	 	
	 	\draw[fill=black] (10,4) circle [radius=0.05]; 
	 	\node[below] at (10,4-0.1) {$v_0$};

	 	\draw[fill=black] (10,5) circle [radius=0.05]; 
	 	\node[above] at (10,5+0.1) {$v_1$};
	 	
	 	\draw[fill=black] (6,5) circle [radius=0.05]; 
	 	\draw[fill=black] (9,5) circle [radius=0.05]; 
	 		\draw[fill=black] (6,4) circle [radius=0.05]; 
	 	\draw[fill=black] (9,4) circle [radius=0.05]; 
	 	\draw[color=red,thick] (5,4)--(2,4)--(2,5)--(6,5);
	 	\draw[color=black,thick] (6,5)--(6,4)--(8,4)--(8,5)--(9,5);
	 	\draw[color=red,thick](9,5)--(13,5)--(13,4)--(10,4);

	 	\draw[fill=black] (5,0) circle [radius=0.05]; 
	 	\node[below] at (5,-0.1) {$u_0$};
	 	\draw[fill=black] (5,1) circle [radius=0.05]; 
	 	\node[above] at (5,1+0.1) {$u_1$};
	 	
	 	\draw[fill=black] (10,0) circle [radius=0.05]; 
	 	\node[below] at (10,-0.1) {$v_0$};

	 	\draw[fill=black] (10,1) circle [radius=0.05]; 
	 	\node[above] at (10,1+0.1) {$v_1$};
	 	
	 		\draw[fill=black] (6,1) circle [radius=0.05]; 
	 	\draw[fill=black] (9,1) circle [radius=0.05]; 
	 		\draw[fill=black] (6,0) circle [radius=0.05]; 
	 	\draw[fill=black] (9,0) circle [radius=0.05]; 
	 	\draw[color=blue,thick] (5,0)--(2,0)--(2,1)--(5,1)--(6,1); 
	 	\draw[color=black,thick] (6,1)--(7,1)--(7,0)--(8,0)--(8,1)--(9,1)--(9,0);
	 	\draw[color=blue,thick] (9,0)--(13,0)--(13,1)--(10,1);

	 	\end{tikzpicture}
	 	\caption{Typical paths in $\mathscr{S}_5(u_0,v_0)$ (red+black+red) and $\mathscr{S}_5(u_0,v_1)$ (blue+black+blue) in the case $v>u+2$.}
	 	\label{fig: SAW on ladder 2}
	 \end{figure*}
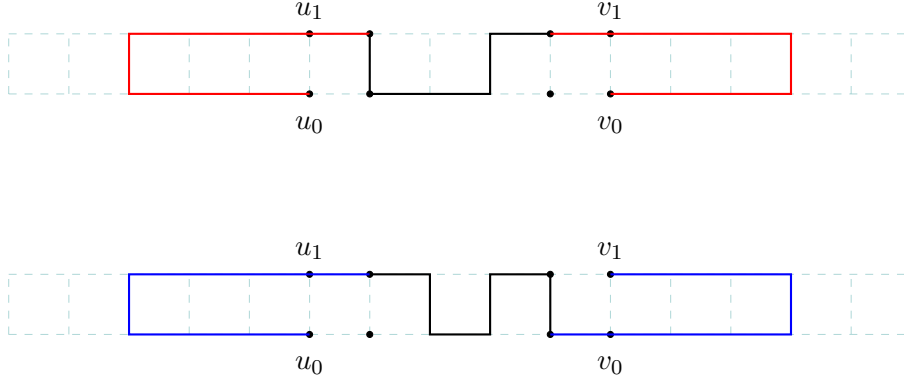	 
	 
\end{enumerate}

\end{proof}

\begin{remark}\label{rem: self-avoiding walk}
From the proof of Proposition \ref{prop: SAW on ladder}, we obtain the following results for self-avoiding walks on \( G \times K_2 \), where \( G \) is a finite, connected, simple graph:

\begin{enumerate}
	\item If \( \min\{\deg_G(u), \deg_G(v)\} = 1 \), then \( \big| \mathscr{S}_5(u_0, v_0) \big| = \big| \mathscr{S}_5(u_0, v_1) \big| = 0 \), and hence \( \big| \mathscr{S}(u_0, v_0) \big| = \big| \mathscr{S}(u_0, v_1) \big| \). In particular, this applies to the case where \( G \) is a star graph.
	
	\item If \( e = (u, v) \) is a cut-edge of \( G \) with \( \deg_G(u), \deg_G(v) \geq 2 \), then \( \big| \mathscr{S}_5(u_0, v_0) \big| \geq 1 > 0 = \big| \mathscr{S}_5(u_0, v_1) \big| \), and hence \( \big| \mathscr{S}(u_0, v_0) \big| > \big| \mathscr{S}(u_0, v_1) \big| \).
\end{enumerate}

\end{remark}

From Remark~\ref{rem: self-avoiding walk} we observe that cut edges may pose an obstacle to \cref{ques: SAW on bunkbed} having an affirmative answer, and we do not know whether this is the only type of obstacle.

\begin{question}\label{ques: modified ques for SAW on bunkbed}
	Consider self-avoiding walks on a bunkbed graph \( G \times K_2 \), where \( G \) is a finite, connected, simple graph. Does the inequality
	\[
	\big| \mathscr{S}(u_0, v_0) \big| \leq \big| \mathscr{S}(u_0, v_1) \big|
	\]
	hold whenever \( u \neq v \in V(G) \) such that \( (u, v) \) is not a cut-edge of \( G \)?
\end{question}

\subsection{Proof of Theorem~\ref{thm: SAW on complete}}
Recall that \( \mathscr{S}_5(u_0, v_i) \) denotes the subset of self-avoiding walks on the bunkbed graph from \( u_0 \) to \( v_i \), using \( u_1 \) and \( v_{1-i} \), but not passing through the two vertical edges \( u_0 u_1 \) and \( v_0 v_1 \).

\begin{lemma}\label{lem: formula for the number of SAW on complete graphs times K2}
Consider self-avoiding walks on the bunkbed graph \( K_n \times K_2 \). Define \( A_n := \big| \mathscr{S}_5(u_0, v_0) \big| \) and \( B_n := \big| \mathscr{S}_5(u_0, v_1) \big| \).

For \( n \geq 3 \), the following expressions hold:
\begin{eqnarray}\label{eq: number of SAW from u0 to v0 in S5}
A_n &=& \sum_{k=1}^{\lfloor \frac{n-2}{2} \rfloor} (2k)! \binom{n-2}{2k} \Bigg[ \sum_{t=0}^{n-2k-2} \frac{(n-2k-2)!}{(n-2k-2-t)!} \binom{t+k}{k} \nonumber \\
&& \quad \quad \times \sum_{t=0}^{n-2k-2} \frac{(n-2k-2)!}{(n-2k-2-t)!} \binom{t+k-1}{k-1} (t+k)(t+k+1) \Bigg],
\end{eqnarray}
and
\begin{equation}\label{eq: number of SAW from u0 to v1 in S5}
B_n = \sum_{k=0}^{\lfloor \frac{n-3}{2} \rfloor} (2k+1)! \binom{n-2}{2k+1} \left[ \sum_{t=0}^{n-2k-3} \frac{(n-2k-3)!}{(n-2k-3-t)!} \binom{t+k}{k} (t+k+1) \right]^2,
\end{equation}
where we adopt the conventions \( \sum_{k=1}^{0} = 0 \) and \( \binom{0}{0} = 1 \).

\end{lemma}
\begin{proof}
	
	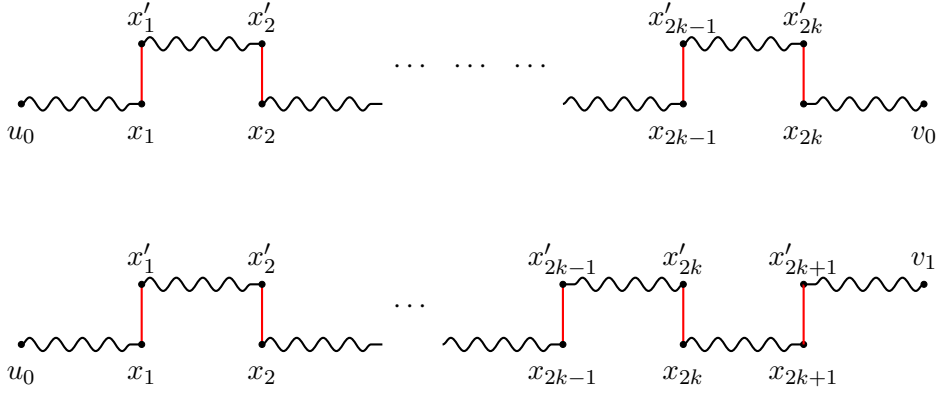
\begin{figure*}[h!]
		\centering
		
		\begin{tikzpicture}[scale=0.8, text height=1.5ex,text depth=.25ex]

		
		\draw[fill=black] (0,4) circle [radius=0.05]; 
		\node[below] at (0,4-0.1) {$u_0$};

		\draw[fill=black] (15,4) circle [radius=0.05]; 
		\node[below] at (15,4-0.1) {$v_0$};

		\node[below] at (6.5,5){$\cdots$};
		\node[below] at (7.5,5){$\cdots$};
		\node[below] at (8.5,5){$\cdots$};
		
		\draw[color=black,thick,snake=snake] (0,4)--(2,4);
		\draw[color=red,thick] (2,4)--(2,5);
		\draw[fill=black] (2,4) circle [radius=0.05]; 
		\node[below] at (2,4-0.1) {$x_1$};
		\draw[fill=black] (2,5) circle [radius=0.05]; 
		\node[above] at (2,5+0.1) {$x_1'$};

		\draw[color=black,thick,snake=snake] (2,5)--(4,5);
		\draw[color=red,thick] (4,5)--(4,4);
		\draw[fill=black] (4,4) circle [radius=0.05]; 
		\node[below] at (4,4-0.1) {$x_2$};
		\draw[fill=black] (4,5) circle [radius=0.05]; 
		\node[above] at (4,5+0.1) {$x_2'$};

		\draw[color=black,thick,snake=snake] (4,4)--(6,4);
		\draw[color=black,thick,snake=snake] (9,4)--(11,4);
		\draw[color=red,thick] (11,5)--(11,4);
		\draw[fill=black] (11,4) circle [radius=0.05]; 
		\node[below] at (11,4-0.1) {$x_{2k-1}$};
		\draw[fill=black] (11,5) circle [radius=0.05]; 
		\node[above] at (11,5+0.1) {$x_{2k-1}'$};

		\draw[color=black,thick,snake=snake] (11,5)--(13,5);	
		
		\draw[color=red,thick] (13,5)--(13,4);	
		
		\draw[fill=black] (13,4) circle [radius=0.05]; 
		\node[below] at (13,4-0.1) {$x_{2k}$};
		\draw[fill=black] (13,5) circle [radius=0.05]; 
		\node[above] at (13,5+0.1) {$x_{2k}'$};	
		
		\draw[color=black,thick,snake=snake] (15,4)--(13,4);

		\draw[fill=black] (0,0) circle [radius=0.05]; 
		\node[below] at (0,0-0.1) {$u_0$};

		\node[below] at (6.5,1){$\cdots$};

		\draw[color=black,thick,snake=snake] (0,0)--(2,0);
		\draw[color=red,thick] (2,0)--(2,1);
		\draw[fill=black] (2,0) circle [radius=0.05]; 
		\node[below] at (2,-0.1) {$x_1$};
		\draw[fill=black] (2,1) circle [radius=0.05]; 
		\node[above] at (2,1+0.1) {$x_1'$};

		\draw[color=black,thick,snake=snake] (2,1)--(4,1);
		\draw[color=red,thick] (4,1)--(4,0);
		\draw[fill=black] (4,0) circle [radius=0.05]; 
		\node[below] at (4,-0.1) {$x_2$};
		\draw[fill=black] (4,1) circle [radius=0.05]; 
		\node[above] at (4,1+0.1) {$x_2'$};

		\draw[color=black,thick,snake=snake] (4,0)--(6,0);
		\draw[color=black,thick,snake=snake] (7,0)--(9,0);
		\draw[color=red,thick] (9,1)--(9,0);
		\draw[fill=black] (9,0) circle [radius=0.05]; 
		\node[below] at (9,-0.1) {$x_{2k-1}$};
		\draw[fill=black] (9,1) circle [radius=0.05]; 
		\node[above] at (9,1+0.1) {$x_{2k-1}'$};

		\draw[color=black,thick,snake=snake] (11,1)--(9,1);	
		
		\draw[color=red,thick] (11,1)--(11,0);	
		
		\draw[fill=black] (11,0) circle [radius=0.05]; 
		\node[below] at (11,-0.1) {$x_{2k}$};
		\draw[fill=black] (11,1) circle [radius=0.05]; 
		\node[above] at (11,1+0.1) {$x_{2k}'$};	
		
		\draw[color=black,thick,snake=snake] (11,0)--(13,0);	
		
		\draw[fill=black] (13,0) circle [radius=0.05]; 
		\node[below] at (13,-0.1) {$x_{2k+1}$};
		\draw[fill=black] (13,1) circle [radius=0.05]; 
		\node[above] at (13,1+0.1) {$x_{2k+1}'$};	
		
		\draw[color=red,thick] (13,1)--(13,0);		
		
		\draw[color=black,thick,snake=snake] (15,1)--(13,1);	
		
		\draw[fill=black] (15,1) circle [radius=0.05]; 
		\node[above] at (15,1+0.1) {$v_1$};

		\end{tikzpicture}
		\caption{A systematic illustration of typical paths in $\mathscr{S}_5(u_0,v_0)$  and $\mathscr{S}_5(u_0,v_1)$.}
		\label{fig: SAW on Kn times K2}
	\end{figure*}

	We prove the formula \eqref{eq: number of SAW from u0 to v0 in S5}, with the other formula \eqref{eq: number of SAW from u0 to v1 in S5} being similar and omitted.

First, for a path \( w \in \mathscr{S}_5(u_0, v_0) \), it must use an even number of vertical edges, say \( 2k \) vertical edges. Since \( u_1, v_1 \in w \), we must have \( k \geq 1 \). Clearly, \( 2k \leq n-2 \), so \( 1 \leq k \leq \left\lfloor \frac{n-2}{2} \right\rfloor \). 
Fix the number of vertical edges. There are \( \binom{n-2}{2k} \) ways to choose the endpoints in the bottom copy of \( K_n \) for these vertical edges, and \( (2k)! \) ways to arrange them in order, say \( x_1, \ldots, x_{2k} \). Let \( x_i' \) denote the corresponding endpoints in the top copy. These contribute to the term \( (2k)! \binom{n-2}{2k} \) in \eqref{eq: number of SAW from u0 to v0 in S5}.

	Now, fixing the number of vertical edges and their endpoints as described above, we proceed to count the number of ways to connect them in a self-avoiding manner, as illustrated in Fig.~\ref{fig: SAW on Kn times K2}.

For \( i \in \{0, 1, \ldots, k\} \), let \( j_i \) denote the number of points in the segment from \( x_{2i} \) to \( x_{2i+1} \) in \( w \) (excluding \( x_{2i} \) and \( x_{2i+1} \)). Moreover, we define \( x_0 = u_0 \) and \( x_{2k+1} = v_0 \). Then, we have \( j_i \geq 0 \) and \( t := \sum_{i=0}^{k} j_i \leq n - 2 - 2k \).
There are \( \binom{n - 2k - 2}{j_0, \ldots, j_k, n - 2k - 2 - t} \) ways to choose the  \( j_0, \ldots, j_k \)  distinct vertices from the remaining \( n - 2k - 2 \) vertices (i.e., \( V(K_n) \setminus \{ u_0, v_0, x_1, \ldots, x_{2k} \} \)), and \( \prod_{i=0}^{k} (j_i)! \) ways to arrange them in order. These contribute the term
\[
\sum_{\substack{j_0, \ldots, j_k \geq 0 \\ j_0 + \cdots + j_k \leq n - 2k - 2}} \binom{n - 2k - 2}{j_0, \ldots, j_k, n - 2k - 2 - t} \cdot \prod_{i=0}^{k} (j_i)! = \sum_{t=0}^{n - 2k - 2} \frac{(n - 2k - 2)!}{(n - 2k - 2 - t)!} \binom{t + k}{k},
\]
where we use the fact that
\[
\big| \{ (j_0, \ldots, j_k) : j_i \geq 0, j_0 + \cdots + j_k = t \} \big| = \binom{t+k}{k}.
\]

For the segments in the top copy, the situation is similar. Let \( j_i' \) denote the number of points (distinct from \( u_1, v_1 \)) in the segment from \( x_{2i-1}' \) to \( x_{2i}' \) for \( i = 1, \ldots, k \). Then, as before, there are
\[
\sum_{\substack{j_1', \ldots, j_k' \geq 0 \\ j_1' + \cdots + j_k' \leq n - 2k - 2}} \binom{n - 2k - 2}{j_1', \ldots, j_k', n - 2k - 2 - t} \cdot \prod_{i=1}^{k} (j_i')! = \sum_{t=0}^{n - 2k - 2} \frac{(n - 2k - 2)!}{(n - 2k - 2 - t)!} \binom{t + k - 1}{k - 1},
\]
ways of choosing these points and arranging them in order. Finally, we need to place \( u_1 \) and \( v_1 \) into some of the segments. Clearly, there are \( (j_1' + 1) + \cdots + (j_k' + 1) = t + k \) ways to insert \( u_1 \). After inserting \( u_1 \), there are \( t + k + 1 \) ways to insert \( v_1 \). Hence, combining all these considerations, we obtain the final term
\[
\sum_{t=0}^{n - 2k - 2} \frac{(n - 2k - 2)!}{(n - 2k - 2 - t)!} \binom{t + k - 1}{k - 1} (t + k)(t + k + 1)
\]
in \eqref{eq: number of SAW from u0 to v0 in S5}.
\end{proof}

\begin{lemma}\label{lem: number of SAW estimates}
Consider self-avoiding walks on a bunkbed graph \( K_n \times K_2 \) and define \( A_n := \left|\mathscr{S}_5(u_0, v_0)\right| \), \( B_n := \left|\mathscr{S}_5(u_0, v_1)\right| \). Then, the following asymptotic behavior holds:

\begin{equation}\label{eq: asymptotic formula for A_n}
A_n = \left(1 + o(1)\right) e^2 \cdot \sum_{m=0}^{\infty} \frac{1}{(m!)^2(m+1)} \cdot (n-3) \big[(n-2)!\big]^2
\end{equation}
and
\begin{equation}\label{eq: asymptotic formula for B_n}
B_n = \left(1 + o(1)\right) e^2 \cdot \sum_{m=0}^{\infty} \frac{1}{(m!)^2} \cdot (n-2) \big[(n-2)!\big]^2
\end{equation}

In particular, there exists \( N > 0 \) such that \( A_n < B_n \) for all \( n \geq N \).

\end{lemma}
\begin{proof}
	Write 
	\begin{align*}
	a_{t,k}&=\frac{(n-2k-2)!}{(n-2k-2-t)!}\binom{t+k}{k}\\
	b_{t,k}&=\frac{(n-2k-2)!}{(n-2k-2-t)!}\binom{t+k-1}{k-1}(t+k)(t+k+1)\\
	c_{t,k}&=\frac{(n-2k-3)!}{(n-2k-3-t)!}\binom{t+k}{k}(t+k+1).
	\end{align*}
	and 
	\begin{align*}
	p_{k}&=(2k)!\binom{n-2}{2k}\sum_{t=0}^{n-2k-2}a_{t,k}\sum_{t=0}^{n-2k-2}b_{t,k}\\
	q_k&=(2k+1)!\binom{n-2}{2k+1}\bigg[\sum_{t=0}^{n-2k-3}c_{t,k}\bigg]^2.
	\end{align*}
	Note that the quantities $a_{t,k},b_{t,k},c_{t,k},p_k$ and $q_k$ depends on $n$. 
	
	Then, by Lemma \ref{lem: formula for the number of SAW on complete graphs times K2}, one has
	\[
	A_n = \sum_{k=1}^{\left\lfloor \frac{n-2}{2} \right\rfloor} p_k
	\quad \text{and} \quad
	B_n = \sum_{k=0}^{\left\lfloor \frac{n-3}{2} \right\rfloor} q_k.
	\]
To establish the asymptotic behavior in this lemma, we rely on the following key estimates:
\begin{itemize}
    	\item[(i)] For a fixed $k$, as $n\to\infty$, we have 
    	\be\label{eq: sum of ck}
    	\sum_{t=0}^{n-2k-3}c_{t,k}=e(1+o(1))\cdot c_{n-2k-3,k}\,.
    	\ee 
    	Similarly 
    	\be\label{eq: sum of ak and bk}
    	\sum_{t=0}^{n-2k-2}a_{t,k}=e(1+o(1))\cdot a_{n-2k-2,k}\,,
    \quad \quad 
    		\sum_{t=0}^{n-2k-2}b_{t,k}=e(1+o(1))\cdot b_{n-2k-2,k}\,.
    	\ee

    	\item[(ii)] 
    	 For any $k$ and $n$ satisfying $0\le k\leq \frac{n-5}{2}$, we have 
    	\be\label{eq: qk ratio bound}
    	\frac{q_{k+1}}{q_k}\leq \frac{e^2}{(k+1)^2}\,,
    	\ee
    	Moreover, for fixed $k\ge 0$, as $n\to\infty$,  we have 
    	\be\label{eq: qk ratio asym}
    	\frac{q_{k+1}}{q_k}=(1+o(1))\frac{1}{(k+1)^2}\,.
    	\ee

    	\item[(iii)]
    	For any $k$ and $n$ satisfying $1\le k\leq \frac{n-4}{2}$, we have 
    	\be\label{eq: pk ratio bound}
    	\frac{p_{k+1}}{p_k}\leq \frac{e^2}{k(k+1)}\,,
    	\ee
    	Moreover, for fixed $k\ge 1$, as $n\to\infty$, we have 
    	\be\label{eq: pk ratio asym}
    	\frac{p_{k+1}}{p_k}=(1+o(1))\frac{1}{k(k+1)}\,.
    	\ee
    	
    \end{itemize}

We first establish the asymptotic behavior given in \eqref{eq: asymptotic formula for A_n} and \eqref{eq: asymptotic formula for B_n}, assuming these estimates  \eqref{eq: sum of ck}---\eqref{eq: pk ratio asym} hold. 

For any fixed $\epsilon > 0$, by \eqref{eq: qk ratio bound}, there exist sufficiently large constants $N_1$ and  $M$, depending only on \( \epsilon \), such that:  
\begin{itemize}
	\item \(\displaystyle \sum_{k=0}^{M} \frac{1}{(k!)^2} \geq (1-\epsilon) \sum_{k=0}^{\infty} \frac{1}{(k!)^2},\)  
	\item for any \( n \geq N_1 \), we have  
	\[
	\sum_{k=0}^{M} q_k \leq B_n \leq (1+\epsilon) \sum_{k=0}^{M} q_k\,.
	\]
\end{itemize}

Next, by \eqref{eq: qk ratio asym}, there exists \( N_2 \) such that for \( n \geq N_2 \), we have  
\[
(1-\epsilon) q_0 \sum_{k=0}^{M} \frac{1}{(k!)^2} \leq \sum_{k=0}^{M} q_k \leq (1+\epsilon) q_0 \sum_{k=0}^{M} \frac{1}{(k!)^2}.
\]
Thus, for \( n \geq \max(N_1, N_2) \), it follows that  
\[
(1-\epsilon)^2 \left( \sum_{k=0}^{\infty} \frac{1}{(k!)^2} \right) q_0 \leq B_n \leq (1+\epsilon) \left( \sum_{k=0}^{\infty} \frac{1}{(k!)^2} \right) q_0.
\]
By \eqref{eq: sum of ck}, we have  
\[
q_0 = e^2 (1+o(1)) (n-2) \cdot \big[(n-2)!\big]^2.
\]  
Substituting this into the expression above, we obtain the desired asymptotic behavior \eqref{eq: asymptotic formula for B_n} for \( B_n \).  

The asymptotic behavior \eqref{eq: asymptotic formula for A_n} follows similarly using the  estimates \eqref{eq: sum of ak and bk}, \eqref{eq: pk ratio bound} and \eqref{eq: pk ratio asym} instead, so we omit the details.

Now it remains to prove the  estimates \eqref{eq: sum of ck}---\eqref{eq: pk ratio asym}.

	Observe that for fixed \( n \) and \( k \), the quantity \( c_{t,k} \) is increasing in \( t \), and for \( 0 \leq t \leq n-2k-4 \), we have
	\be\label{eq: ratio of ctk between t+1 and t}
	\frac{c_{t+1,k}}{c_{t,k}}=(n-2k-3-t)\cdot \frac{t+k+2}{t+1}.
	\ee
Thus, using the inequality \( \frac{c_{t+1,k}}{c_{t,k}} \geq n-2k-3-t \) for \( t \in [0, n-2k-4] \), we obtain the following inequality for \( 0 \leq k \leq \left\lfloor \frac{n-3}{2} \right\rfloor \):
	\be\label{eq: upper bound ctk}
	\sum_{t=0}^{n-2k-3}c_{t,k}\leq (1+\frac{1}{1!}+\frac{1}{2!}+\frac{1}{ 3!}+\frac{1}{4!}+\cdots )c_{n-2k-3,k}=ec_{n-2k-3,k}
	\ee

Note that for \( 0 \leq k \leq \left\lfloor \frac{n-3}{2} \right\rfloor - 1 \), we have  
\begin{equation} \label{eq: ratio of y_k}
\frac{q_{k+1}}{q_k} = \frac{(n-2k-3)!}{(n-2k-5)!} \cdot \frac{\left( \sum_{t=0}^{n-2k-5} c_{t,k+1} \right)^2}{\left( \sum_{t=0}^{n-2k-3} c_{t,k} \right)^2}.
\end{equation}
Thus, applying \eqref{eq: upper bound ctk} and the trivial lower bound \( \sum_{t=0}^{n-2k-3} c_{t,k} \geq c_{n-2k-3,k} \), we obtain \eqref{eq: qk ratio bound}:  
\begin{align*}
\frac{q_{k+1}}{q_k} 
&\leq (n-2k-3)(n-2k-4) \cdot \frac{e^2 c_{n-2k-5,k+1}^2}{c_{n-2k-3,k}^2} \nonumber\\
&= \frac{e^2}{(k+1)^2} \cdot \frac{(n-2k-3)(n-2k-4)}{(n-k-2)^2} \leq \frac{e^2}{(k+1)^2}. 
\end{align*}

For fixed \( k \), as \( n \to \infty \), using \eqref{eq: ratio of ctk between t+1 and t}, we obtain  
\begin{align*}
\sum_{t=0}^{n-2k-3} c_{t,k} 
&= c_{n-2k-3,k} \cdot \bigg( 1 + \frac{n-2k-3}{n-k-2} + \frac{1}{2!} \cdot \frac{(n-2k-3)(n-2k-4)}{(n-k-2)(n-k-3)} \nonumber\\
&\quad + \frac{1}{3!} \cdot \frac{(n-2k-3)(n-2k-4)(n-2k-5)}{(n-k-2)(n-k-3)(n-k-4)} + \cdots \bigg) \nonumber\\
&\geq e(1-o(1)) c_{n-2k-3,k}.
\end{align*}  
This, together with \eqref{eq: upper bound ctk}, implies \eqref{eq: sum of ck}. Substituting \eqref{eq: sum of ck} into \eqref{eq: ratio of y_k}, we obtain \eqref{eq: qk ratio asym}:  
\[
\frac{q_{k+1}}{q_k} = (1+o(1))(n-2k-3)(n-2k-4) \cdot \frac{c_{n-2k-5,k+1}^2}{c_{n-2k-3,k}^2} = (1+o(1))\frac{1}{(k+1)^2}.
\]
The estimates \eqref{eq: sum of ak and bk}, \eqref{eq: pk ratio bound}, and \eqref{eq: pk ratio asym} can be derived in a similar manner, and we omit the details.
\end{proof}


\begin{proof}[Proof of Theorem~\ref{thm: SAW on complete}]
	The case \( n = 2 \) is trivial to verify. For \( n = 3,4,5 \), we compute from Lemma \ref{lem: formula for the number of SAW on complete graphs times K2} that  
	\[
	A_3 = 0, \quad B_3 = 1, \quad A_4 = 4, \quad B_4 = 18, \quad A_5 = 144, \quad B_5 = 387.
	\]
	Hence, by Lemma \ref{lem: bijections between some subsets of SAW}, we obtain the desired inequality  
	\[
	\big|\mathscr{S}(u_0,v_0)\big| < \big| \mathscr{S}(u_0,v_1) \big|
	\]
	for \( n \in \{3,4,5\} \).  
	
	For sufficiently large \( n \), Lemma \ref{lem: number of SAW estimates} ensures that \( A_n < B_n \), which again implies the desired inequality  
	\[
	\big|\mathscr{S}(u_0,v_0)\big| < \big| \mathscr{S}(u_0,v_1) \big|
	\]  
	by Lemma \ref{lem: bijections between some subsets of SAW}.
\end{proof}


\bibliographystyle{plain} 
\bibliography{bunkbed_ref}

\end{document}